\documentclass [twoside,reqno,11pt]{amsart}

 \usepackage[a4paper, total={6in, 8in}]{geometry}
\usepackage[usenames]{color}
\usepackage[abbrev,nobysame,alphabetic]{amsrefs}
\usepackage{todonotes}
\usepackage{mathtools}  
\mathtoolsset{showonlyrefs}

\usepackage{amsmath,pdfsync,verbatim,graphicx,epstopdf,enumerate}
\usepackage[colorlinks=true]{hyperref}
\usepackage{cancel}
\usepackage{accents}
\hypersetup{linkcolor=blue,citecolor=red}
\usepackage{mathtools}  
\newcommand{\D}{\mathrm{d}}

\numberwithin{equation}{section}

\usepackage[english]{babel}
\usepackage{epsfig,graphics, graphicx}
\usepackage{subcaption, comment, bm}
\usepackage[percent]{overpic}
\usepackage{float}
\usepackage{color}
\usepackage[]{amsfonts}
\usepackage{amsopn}
\usepackage{amsmath,amsthm,amssymb}
\usepackage{relsize}
\usepackage[]{fancyhdr}
\usepackage[]{graphicx,wrapfig}
\usepackage{enumitem}
\usepackage{array}
\usepackage[T1]{fontenc}
\newtheorem{theorem}{Theorem}[section]

\newtheorem{lemma}[theorem]{Lemma}
\newtheorem{proposition}[theorem]{Proposition}

\theoremstyle{definition}
\newtheorem{definition}[theorem]{Definition}

\newtheorem{remark}[theorem]{Remark}

\usepackage{amsmath,pdfsync,verbatim,graphicx,epstopdf,enumerate}

\newcommand{\FR}{\mathbb{R}}

\newcommand{\Sb}{\mathbb{S}}

\title[Microlocal inversion]{Microlocal inversion of a restricted mixed ray transform for second-order tensor fields in $\mathbb{R}^3$}
\author[C. Thakkar]{Chandni Thakkar}
\address{Department of Mathematics, IIT Gandhinagar, Gujarat, India}
\email{thakkar\_chandni@iitgn.ac.in}
\begin{document}
\begin{abstract}
In this article, we study a restricted mixed ray transform acting on second-order tensor fields in 3-dimensional Euclidean space and prove the invertibility of this integral transform using microlocal techniques. Here, the mixed ray transform is restricted over lines passing through a fixed curve $\gamma$ in $\mathbb{R}^3$ satisfying certain geometric conditions. The main theorem of the article shows that a second-order tensor field can be recovered from its restricted mixed-ray transform up to the kernel of the transform, a smoothing term, and a known singular term.
\end{abstract}
	\subjclass[2020]{35A22,35S30,46F12}
	\keywords{Mixed ray transform, microlocal analysis}
\maketitle
\section{Introduction}

The mixed ray transform (MiRT) was introduced by Sharafutdinov in \cite{Sharafutdinov_1994}*{Chapter 7} to study the geometrical aspects of quasi-isotropic elastic media. It was shown that compared to the classical isotropic media, the formulas for zero approximation, in this case, had some additional features that require the study of the MiRT. Also, a detailed description of the appearance of MiRT in anisotropic perturbation of the Dirichlet-to-Neumann map of an isotropic elastic wave equation on a smooth and bounded domain in three-dimensional Euclidean space was shown in \cite{Uniqueness_stability_MRT}, along with some stability estimates for the normal operator in the case of $1 + 1$ and $2 + 2$ tensor fields. In \cite{kernel_MRT}, an explicit kernel description for the MiRT on 2-dimensional Riemannian manifolds was given. Recently in \cite{UCP_MiRT}, an inversion formula for the MiRT for $(k + \ell)$-ordered tensor field in $\mathbb{R}^2$ was given along with the range characterization for the transform and some unique continuation results. In this paper, we study microlocal inversion of a restricted MiRT acting on 2-tensor fields in $\mathbb{R}^3$.

Since the work on generalized Radon transforms in the framework of Fourier integral operators by Guillemin \cite{Guillemin_generalised_radon} and Guillemin-Sternberg \cite{Guillemin_Sternberg_Microlocal}, there have been many important results proved using microlocal analysis techniques; for this, we refer to \cites{Quinto_Microlocal_genralized_Radon, Venky_Quinto_Microlocal_analysis, Uhlmann_Greenleaf_Nonlocal_inversion, Alain_Microlocal_analysis, Microlocal_2018,  Microlocal_2021, Katsevich_microlocal_analysis, Uhlmann_microlocal_scalar_2003, Microlocal_doppler_transform, Sharafutdinov_Uhlmann_2005, Uhlamnn_Stefanov_stability_estimates, Uhlmann_Stefanov_2005, Uhlmann_Stefanov_2008, Uhlmann_Vasy_2016} and the references therein. Certain support theorems are also proved using microlocal analysis; for instance, see \cites{Venky_support_function_2009, Boman_Quinto_support_theorem_Radon, Boman_Quinto_support_theorem_3DRadon, Venky_Stefanov_support_theorem, Anuj_Rohit_support_theorems}. In \cite{Uhlmann_Greenleaf_Nonlocal_inversion}, Greenleaf and Uhlmann studied a restricted data problem for ray transforms acting on functions in the Riemannian geometry setting. Further, in \cite{Microlocal_doppler_transform}, microlocal inversion for the Doppler transform in 3-dimensional Euclidean space was studied. Recently in \cite{Microlocal_2018} and \cite{Microlocal_2021}, the microlocal inversion of a restricted longitudinal ray transform in $\mathbb{R}^n$ and restricted transverse ray transform in $\mathbb{R}^3$ for symmetric tensor fields was derived. This paper aims to prove a similar result for the mixed ray transform acting on 2-tensor fields in $\mathbb{R}^3$. We define the transform for tensor fields of arbitrary order $k + \ell$, where $k, \ell = 1, 2, 3, \dots$. The inversion problem for arbitrary order tensor fields is quite hard, as observed in earlier works as well; please see Remark \ref{remark: generalization} (below) for a comment about the complexity in the general case. Our focus in this work is to solve this problem in the case of 2-tensor fields.

Let $C^\infty(T^2)$ denote the space of smooth 2-tensor fields in $\mathbb{R}^3$ and $C_c^\infty(T^2) \subset C^\infty(T^2)$ be the space of smooth compactly supported 2-tensor fields in $\mathbb{R}^3$. Any element $f \in C_c^\infty(T^2)$ can be written as 
\[f (x) = f_{ij} (x) \D x^i \D x^j,\]
where $f_{ij}$ are smooth and compactly supported functions in $\mathbb{R}^3$. Throughout this article, Einstein summation convention will be followed for repeating indices. Let $\mathbb{S}^{2}$ denote the unit sphere in $\mathbb{R}^3$, then the space of straight lines in $\mathbb{R}^3$ can be parameterized by:

\[T\mathbb{S}^{2} = \left\{(x,\xi) \in \mathbb{R}^3 \times \mathbb{S}^{2}: \left<x, \xi\right> = 0\right\}.\]

\noindent A pair $(x, \xi)$ in $T\mathbb{S}^{2}$ represents a unique line in $\mathbb{R}^3$ passing through $x$ and in the direction of $\xi$. 
\begin{definition}
    The mixed ray transform $\mathcal{M}$ acting on $(k + \ell)$-ordered tensor field, which is symmetric with respect to first $k$ and last $\ell$ indices, is defined as follows:
    \begin{equation}
        \mathcal{M} f (x, \xi) = \int_{-\infty}^\infty f_{i_1 \dots i_k j_1 \dots j_\ell} (x + t\xi) \xi^{i_1} \dots \xi^{i_k} \eta^{j_1} \dots \eta^{j_\ell} \,dt,  
    \end{equation}
    where $\eta$ is any vector orthogonal to $\xi$.
\end{definition}

\noindent In this article, we are concerned with 2-tensor fields, and hence we define an equivalent version of the above definition. For any vector $\xi \in \mathbb{S}^{2}$, there exists $0 \leq \alpha < \pi$ and $0 \leq \beta < 2 \pi$ such that $\xi$ can be represented as follows:
\[\xi = (\sin \alpha \cos \beta, \sin \alpha \sin \beta, \cos \alpha).\]
Also, the set $\left\{\xi, \xi_\alpha, \xi_\beta\right\}$ with $\xi_\alpha = (\cos \alpha \cos \beta, \cos \alpha \sin \beta, - \sin \alpha)$ and $\xi_\beta = ( - \sin \beta, \cos \beta, 0)$
forms an orthonormal basis of $\mathbb{R}^3$. Now we give another version of the definition of the MiRT using above notations.

\begin{definition}
    The mixed ray transform (MiRT) $\mathcal{M} = (\mathcal{M}_\alpha, \mathcal{M}_\beta) :C_c^{\infty}(T^2) \rightarrow {(C^{\infty} (T\Sb^2))}^2$ is a bounded linear operator which is defined as follows:
    \begin{equation}\label{eq:mixed ray transform}
    \begin{aligned}
        \mathcal{M}_\alpha f (x, \xi) &= \int_{-\infty}^\infty \left<f (x + t \xi), \xi \otimes \xi_\alpha\right> \,dt &= \int_{-\infty}^\infty f_{ij} (x + t \xi)\xi^{i}\xi_\alpha^{j}\,dt,\\ 
        \mathcal{M}_\beta f (x, \xi) &= \int_{-\infty}^\infty \left<f (x + t \xi), \xi \otimes \xi_\beta\right> \,dt  &= \int_{-\infty}^\infty f_{ij} (x + t \xi)\xi^{i}\xi_\beta^{j}\,dt
    \end{aligned}
    \end{equation}
    where $\otimes$ denotes the usual tensor product.
\end{definition}

The space of straight lines in $\mathbb{R}^3$ is 4-dimensional. Hence, inversion of the transform defined on a collection of straight lines to recover a 3-dimensional tensor field is an overdetermined problem. So, the question arises of recovering the tensor field from a 3-dimensional restricted transform data. In this paper, we consider the set of straight lines passing through a fixed curve, satisfying the Kirillov-Tuy condition (see Definition \ref{def: Kirillov-Tuy condition} below), in $\mathbb{R}^3$. The aim here is to recover the wavefront set (that is, singularities) of the unknown tensor field from the knowledge of its restricted MiRT data. The primary motivation to study this problem comes from similar works \cite{Microlocal_doppler_transform, Microlocal_2018, Microlocal_2021, Microlocal_thesis} related to different integral transforms. We closely follow these manuscripts to prove the main theorem of this article. 

The article is organized as follows. In Section \ref{sec: definitions}, we fix some notations and definitions and end it with the statement of the main theorem. Section \ref{section: Principal symbol} is dedicated to the analysis of the principal symbol of the normal operator of the MiRT. Finally, Section \ref{sec: microlocal_inversion} is devoted to the proof of our main theorem.

\section{Preliminaries and statement of the main result}\label{sec: definitions}
This section is dedicated to setting up the notations and defining some important operators for our study. We require the definitions of the operators only for lower-ordered tensor fields, and hence, we have defined accordingly. Their definitions for higher-ordered tensor fields and more details can be found in \cite{Sharafutdinov_1994}. Toward the end of this section, we state the main result of this article.

Let $u = (u_1, u_2, u_3)$ be a vector field in $\mathbb{R}^3$. Then the operator of \textit{inner differentiation} $\D$ (also known as symmetrized derivative) acting on $u$ is a symmetric $2$-tensor field given by:
\begin{equation} \label{def: inner differntiation}
\left(\D u\right)_{ij} = \frac{1}{2} \left[\frac{\partial u_i}{\partial x_j} + \frac{\partial u_j}{\partial x_i}\right].
\end{equation}

\noindent The next two operators are important to understand the kernel of the MiRT: Let $u$ and $w$ be vector and scalar fields, respectively, in $\mathbb{R}^3$. Then the operators $\D'$ and $\lambda$ are defined as follows:

\begin{equation}
    {(\D^\prime u)}_{ij} = \frac{\partial u_j}{\partial x_i} \quad \text{ and } \quad {(\lambda w)}_{ij} = \delta_{ij} w, \quad \text{ where } \quad \delta_{ij} = \begin{cases*}
        1; & $i = j$\\
        0; & $i \neq j.$
        \end{cases*}
\end{equation}

\noindent If $u$ and $w$ are second-order tensor fields, then the dual to the above operators, $\delta'$ and $\mu$, are defined as follows:

\begin{equation}
    {(\delta' u)}_j = \frac{\partial u_{ij}}{\partial x^i} \quad \text{ and } \quad \mu w = \delta^{ij} w_{ij}.
\end{equation}
The following decomposition for a $(k + \ell)$-ordered tensor fields, symmetric with respect to first $k$ and last $\ell$ indices, was proved in \cite[Proposition 3.2]{UCP_MiRT}.

\begin{proposition} \label{thm: tensor field decomposition}
    For any tensor field $f \in C_c^\infty (\mathit{S}^k \times \mathit{S}^\ell)$, there exists $f^s \in C^\infty (\mathit{S}^k \times \mathit{S}^\ell),\  u \in C^\infty (\mathit{S}^{k-1} \times \mathit{S}^{\ell})$ and $w \in C^\infty (\mathit{S}^{k - 1} \times \mathit{S}^{\ell - 1})$ such that:
\begin{equation}\label{tensor field decomposition}
   f = f^s + \D^\prime u + \lambda w,
\end{equation}
where, $\delta'f^s = \mu f^s = 0, \ \mu u = 0$, and $f^s,u \rightarrow 0$ as $|x| \rightarrow \infty$.
\end{proposition}

\noindent It is also known that the tensor fields of the form $\D' u$ and $\lambda w$ lie in the kernel of the MiRT, that is, $\mathcal{M}(\D' u) = 0 = \mathcal{M}(\lambda w)$. Hence, using the MiRT data, we can only hope to recover $f^{s}$. We call $f^{s}$ to be \textit{solenoidal} because $\delta' f^{s} = 0$. As mentioned earlier, we consider the transform data on a collection of lines passing through a fixed curve $\gamma$ in $\mathbb{R}^3$. We now state the required conditions on this curve $\gamma$, starting with the well-known \textit{Kirillov-Tuy condition}. This condition was initially introduced for the scalar case in \cite{Kirillov_1961,Tuy_1983} and later for tensor fields in \cite{Denisjuk_2006,Vertgeim_2000}.

\begin{definition}[Kirillov-Tuy condition] \label{def: Kirillov-Tuy condition}
    Let $B$ be a ball in $\mathbb{R}^3$. Say that the curve $\gamma$ satisfies the Kirillov-Tuy condition of order $2$ if for almost any hyperplane $H$ intersecting the domain $B$, there are points $\gamma_1,\gamma_2,\gamma_3 \in H \cap \gamma$, such that for almost every $x \in H \cap B$, the vectors $x - \gamma_1,x - \gamma_2,x - \gamma_3$ are pairwise independent.
\end{definition}

\begin{remark}
     The following conditions are imposed on the curve $\gamma$ under consideration:
     \begin{enumerate}
         \item It is smooth, regular and without self-intersections.
         \item It satisfies the Kirillov-Tuy condition of order 2.
         \item There is a uniform bound on the number of intersection points of the curve with almost every hyperplane $H$ in $\mathbb{R}^3$.
     \end{enumerate}
\end{remark}

\noindent For the rest of the article, we assume that the transform $\mathcal{M}$ is known for lines passing through curve $\gamma$, satisfying the above conditions. Following \cite{Uhlmann_Greenleaf_Nonlocal_inversion, Microlocal_doppler_transform, Microlocal_2018, Microlocal_2021}, we put certain restrictions on the singular directions that can be potentially recovered by the microlocal approach. For this, we define the following sets:

\begin{align}
\begin{split}
    \Xi &= \left\{(x, \xi) \in T^* B\symbol{92}\left\{0\right\}: \mbox{ there exists at least } 3 \mbox{ points } \gamma(t_i) \in \gamma \cap H(x, \xi); i = 1, 2, 3,\right.\\
    &\hspace{4mm}\left.\mbox{ such that the vectors } \left\{x - \gamma(t_i)\right\}_{i = 1}^3 \mbox{ are pairwise independent} \right\}.\\
    \Xi' &= \left\{(x, \xi) \in \Xi : H(x, \xi) \mbox{ intersects the curve } \gamma \mbox{ transversely} \right\}.\\
    \Xi'' &= \left\{(x, \xi) \in \Xi : H(x, \xi) \mbox{ intersects } \gamma \mbox{ tangentially at points } \left\{\gamma(t_i)\right\} \mbox{ such that } \left<\gamma'' (t_i), \xi\right> \neq 0 \right\}.
\end{split}
\end{align}

In the main theorem, we show that the potentially recoverable singularities lie in the union of the sets $\Xi'$ and $\Xi''$. Before stating the main theorem, we state some results important for microlocal analysis of operator $\mathcal{M}$ and its normal operator $\mathcal{N} = \mathcal{M}^* \mathcal{M}$, where $\mathcal{M}^*$ denotes the $\textit{L}^2-$ adjoint of $\mathcal{M}$. For details about the normal operator of the MiRT, we refer to \cite{UCP_MiRT}. The proofs of the results stated below can be found directly in \cite{Microlocal_thesis, Microlocal_2018}, or can be done by suitable modifications to them.

Let $\mathcal{C}$ denote the set of all lines passing through the curve $\gamma$. Any line $\ell$ in $\mathcal{C}$ can be identified using $t$ (in the domain of $\gamma$) and $\omega \in \mathbb{S}^2$, such that, $\ell = \left\{\gamma(t) + s \omega: s \in \mathbb{R}\right\}$. Further, let

\begin{equation*}
    Z = \left\{(\ell, x) : x \in \ell\right\} \subset \mathcal{C} \times \mathbb{R}^3.
\end{equation*}
be the point-line relation. Then, we get $(t, \omega, s)$ to be a local parametrization of $Z$. As shown in \cite{Microlocal_thesis, Microlocal_2018, Microlocal_2021}, the conormal bundle of $Z$ is given by 

\begin{equation*}
    N^* Z = \left\{(\ell, x; \Gamma, \xi) : (\ell, x) \in Z \text{ and } (\Gamma, \xi)|_{T_{(\ell, x)}Z}\right\}.
\end{equation*}

\noindent $N^* Z$ can be parametrized by $\left\{(t, \omega, s, \Gamma, \xi)\right\}$ with

\begin{equation} \label{eq: xi}
    \xi = z_1 \omega_\alpha + z_2 \omega_\beta \text{ for some } z_1 \text{ and } z_2 \in \mathbb{R}
\end{equation}
and
\begin{equation} \label{eq: gamma}
    \Gamma = \begin{pmatrix}
\Gamma_1\\
\Gamma_2\\
\Gamma_3
\end{pmatrix} = \begin{pmatrix}
-\xi \cdot \gamma'(t)\\
-s z_1\\
-s z_2 \sin \alpha
\end{pmatrix}.
\end{equation}

\begin{lemma}
    The map 
    \begin{equation*}
        \Phi : (t, \alpha, \beta, s, z_1, z_2) \rightarrow (t, \alpha, \beta, \Gamma; x, \xi)
    \end{equation*}
    with $\Gamma$ as in \eqref{eq: gamma}, $\xi$ as in \eqref{eq: xi} and $x = \gamma(t) + s \omega$ gives a local parametrization of $N^* Z$ at the points where $\alpha \neq 0, \pi$.
\end{lemma}

\begin{proposition}
    Both the components of the operator $\mathcal{M}$ are Fourier integral operators of order $-1/2$ with the associated canonical relation $C$ given by $(N^* Z)'$ where $Z = \left\{(\ell, x) : x \in \ell\right\}$. The left and right projections $\pi_L$ and $\pi_R$ from $C$ drop rank simply by 1 on the set

    \begin{equation}
        \Sigma : = \left\{(t, \alpha, \beta, s, z_1, z_2) : \gamma'(t) \cdot \xi = 0\right\},
    \end{equation}
    where $\xi$ is given by \eqref{eq: xi}. The left projection $\pi_L$ has a blowdown singularity along $\Sigma$ and the right projection $\pi_R$ has a fold singularity along $\Sigma$.
\end{proposition}

\begin{lemma}
    The wavefront set of the Schwartz kernel of the normal operator $\mathcal{N}$ satisfies the following:

    \begin{equation*}
        WF(\mathcal{N}) \subset \Delta \cup \Lambda
    \end{equation*}
    where 
    \begin{align}
        \Delta = \left\{(x, \xi; x, \xi) : x = \gamma(t) + s \omega, \xi \in \omega^\perp \symbol{92} \left\{0\right\}\right\}
    \end{align}
    and
    \begin{align}
        \Lambda = \left\{\left(x, \xi; y, \frac{\tau}{\tilde{\tau}} \xi\right) : x = \gamma(t) + \tau \omega, y = \gamma(t) + \tilde{\tau} \omega, \xi \in \omega^\perp \symbol{92} \left\{0\right\}, \gamma'(t)\cdot \xi = 0, \tau \neq 0 \neq \tilde{\tau} \right\}.
    \end{align}
\end{lemma}

\noindent The condition imposed on the curve in the definition of $\Xi''$ entails the clean intersection of the sets $\Delta$ and $\Lambda$. We have 

\begin{align}
    \Delta \cap \Lambda = \left\{(x, \xi; x, \xi) : x = \gamma(t) + s \omega, \xi \in \omega^\perp \symbol{92} \left\{0\right\}, \gamma'(t)\cdot \xi = 0\right\}.
\end{align}

\noindent Hence, $\Delta \cap \Lambda$ is a smooth manifold of codimension $k = 1$ in both $\Delta$ and $\Lambda$.

\begin{lemma}[\cite{Microlocal_thesis}]
    The Lagrangian $\Lambda$ arises as a flowout from the set $\pi_R(\Sigma)$.
\end{lemma}

We want to study the operators $\mathcal{M}$ and $\mathcal{N}$ in terms of $I^{p, l}$ classes of distributions. For more details about the $I^{p, l}$ classes, we refer to \cite{Uhlmann_Greenleaf_Nonlocal_inversion, Melrose_Uhlmann_Lagrangian_intersection, Guillemin_Uhlmann_Oscillatory_integrals}. Here, we mention some important properties of the $I^{p, l}$ class of distributions that are required for further discussion.

Let $\Delta$ and $\Lambda$ be two cleanly intersecting Lagrangians with intersection $\Sigma = \Delta \cap \Lambda$ and let $u \in I^{p, l} (\Delta, \Lambda)$, then

\begin{enumerate}
    \item $WF(u) \subset \Delta \cup \Lambda.$
    \item Microlocally, the Schwartz kernel of $u$ equals the Schwartz kernel of a pseudodifferential operator of order $p + l$ on $\Delta \symbol{92} \Lambda$ and that of a classical Fourier integral operator of order $p$ on $\Lambda \symbol{92} \Delta$.
    \item $I^{p, l} \subset I^{p', l'}$ if $p \leq p'$ and $l \leq l'$.
    \item $\cap_l I^{p, l} (\Delta, \Lambda) \subset I^p(\Lambda)$.
    \item $\cap_p I^{p, l} \subset $ The class of smoothing operators.
    \item The principal symbol $\sigma_0(u)$ on $\Delta \symbol{92} \Sigma$ has the singularity on $\Sigma$ as a conormal distribution of order $l - \frac{k}{2}$, where $k$ is the codimension of $\Sigma$ as a submanifold of $\Delta$ or $\Lambda$.
    \item If the principal symbol $\sigma_0(u) = 0$ on $\Delta \symbol{92} \Sigma$, then $u \in I^{p, l - 1} (\Delta, \Lambda) + I^{p - 1, l} (\Delta, \Lambda)$.
    \item $u$ is said to be elliptic if the principal symbol $\sigma_0(u) \neq 0$ on $\Delta \symbol{92} \Sigma$ if $k \geq 2$, and for $k = 1$, if $\sigma_0(u) \neq 0$ on each connected component of $\Delta \symbol{92} \Sigma$.
\end{enumerate}

Since $\Lambda$ in our case arises as a flowout, the following composition calculus by Antoniano and Uhlmann can be used to construct a relative left parametrix for our operator $\mathcal{N}$.

\begin{theorem}\cite{Uhlmann_composition_calculus} \label{Uhlmann_composition_calculus}
    If $A \in I^{p,l}(\Delta, \Lambda)$ and $B \in I^{p',l'}$, then composition of $A$ and $B$, $A \circ B \in I^{p + p' + \frac{k}{2}, l + l' - \frac{k}{2}} (\Delta, \Lambda)$ and the principal symbol, $\sigma_0 (A \circ B) = \sigma_0(A) \sigma_0(B)$, where, $k$ is the codimension of $\Sigma$ as a submanifold of either $\Delta$ or $\Lambda$.
\end{theorem}

\noindent We now state the main theorem of this article.

\begin{theorem} \label{thm: main theorem}
    Let $\Xi_0 \subseteq \Xi'$ be such that $\overline{\Xi}_0 \subseteq \Xi' \cup \Xi''$ and $K$ be a closed conic subset of $\Xi_0$. Let $\mathcal{E}'_{K} (B) \subset \mathcal{E}' (B)$ denote the space of compactly supported distributions in $B$ whose wavefront set is contained in $K$. Then there exists an operator $\mathcal{B} \in I^{0, 1}(\Delta, \Lambda)$ and an operator $\mathcal{A} \in I^{-1/2}(\Lambda)$ such that for any second-order tensor field $f$ with coefficients in $\mathcal{E}'_{K} (B)$, we have for each $x \in \pi_1 (K)$

    \begin{equation*}
        \mathcal{B} \mathcal{N} f (x) = f^{s} (x) + \mathcal{A} f (x) + \text{ smoothing terms}.
    \end{equation*}
\end{theorem}

We start the proof of the above theorem by stating the principle symbol of the operator $\mathcal{N}$ on the diagonal $\Delta$ away from $\Sigma$ and then using it to construct a relative left parametrix for the operator. The theorem is proved by making suitable changes to the techniques of \cite{Uhlmann_Greenleaf_Nonlocal_inversion, Microlocal_2018, Uhlmann_microlocal_scalar_2003, Microlocal_doppler_transform, Microlocal_2021}.

\section{Principal symbol of the operator $\mathcal{N}f$} \label{section: Principal symbol}
The following proposition gives the expression for the principal symbol matrix of the normal operator $\mathcal{N}$.

\begin{proposition} \label{prop: principal symbol}
    Let $(x, \xi) \in K$. Then the principal symbol matrix $A_0(x, \xi)$ of the operator $\mathcal{N}$ is as follows:

    \begin{equation}
        A_0(x, \xi) = \sum_q \frac{2 \pi \omega_q^{i_1} {(\omega_q)}_{\alpha}^{i_2} \omega_q^{j_1} {(\omega_q)}_{\alpha}^{j_2}}{|\xi| |\gamma' (t_q(\xi_0)\cdot \xi_0)| |\gamma (t_q(\xi_0) - x)|} + \sum_q \frac{2 \pi \omega_q^{i_1} {(\omega_q)}_{\beta}^{i_2} \omega_q^{j_1} {(\omega_q)}_{\beta}^{j_2}}{|\xi| |\gamma' (t_q(\xi_0)\cdot \xi_0)| |\gamma (t_q(\xi_0) - x)|}
    \end{equation}
    Here, $\xi_0$ is the unit vector along $\xi$ direction and $q$ varies over the number of intersection points of the plane $H(x, \xi)$ with the curve $\gamma$.
\end{proposition}

\noindent The proof of the above proposition is similar to \cite{Microlocal_doppler_transform, Microlocal_2018, Microlocal_thesis} and hence we skip the details here. We prove the ellipticity of the symbol on the solenoidal part in the following proposition.
\begin{remark} \label{remark: generalization}
The principal symbol matrix of $\mathcal{N}$ can be computed for tensor fields of arbitrary order $(k + \ell)$ by following exactly similar analysis. The main difficulty comes in proving the ellipticity of this symbol over tensor fields that are trace-free and divergence-free. At this point, we do not know how to incorporate the trace-free conditions (appropriately) along with the other linearly independent conditions to show the ellipticity for the general case. Even for a $2+2$ tensor field, the expressions become quite complicated, and we could not find a way to simplify them to obtain the required ellipticity.
\end{remark}
\begin{proposition} \label{prop: ellipticity proposition}
    Let $f$ be a $2$-tensor field in $\FR^3$ such that $\xi^{i_1} f_{i_1 i_2} = 0$ and $f$ is trace-free, that is, $\sum_{i} f_{ii} = 0$. If $A_0 f = 0$, then $f = 0$.
\end{proposition}
\begin{proof}
    We have $A_0 f = 0$, that is, 
    \begin{equation*}
        \sum_q \frac{2 \pi \omega_q^{i_1} {(\omega_q)}_{\alpha}^{i_2} \omega_q^{j_1} {(\omega_q)}_{\alpha}^{j_2} f_{i_1 i_2}}{|\xi| |\gamma' (t_q(\xi_0)\cdot \xi_0)| |\gamma (t_q(\xi_0) - x)|} + \sum_q \frac{2 \pi \omega_q^{i_1} {(\omega_q)}_{\beta}^{i_2} \omega_q^{j_1} {(\omega_q)}_{\beta}^{j_2} f_{i_1 i_2}}{|\xi| |\gamma' (t_q(\xi_0)\cdot \xi_0)| |\gamma (t_q(\xi_0) - x)|} = 0.
    \end{equation*}
    Multiplying the above equation by $f_{j_1 j_2}$ and adding, we get
    \begin{equation*}
        \sum_q \frac{2 \pi \omega_q^{i_1} {(\omega_q)}_{\alpha}^{i_2} \omega_q^{j_1} {(\omega_q)}_{\alpha}^{j_2} f_{i_1 i_2} f_{j_1 j_2}}{|\xi| |\gamma' (t_q(\xi_0)\cdot \xi_0)| |\gamma (t_q(\xi_0) - x)|} + \sum_q \frac{2 \pi \omega_q^{i_1} {(\omega_q)}_{\beta}^{i_2} \omega_q^{j_1} {(\omega_q)}_{\beta}^{j_2} f_{i_1 i_2} f_{j_1 j_2}}{|\xi| |\gamma' (t_q(\xi_0)\cdot \xi_0)| |\gamma (t_q(\xi_0) - x)|} = 0.
    \end{equation*}
    This implies
    \begin{equation*}
        \sum_q \frac{2 \pi {(\omega_q^{i_1} {(\omega_q)}_{\alpha}^{i_2} f_{i_1 i_2})}^2}{|\xi| |\gamma' (t_q(\xi_0)\cdot \xi_0)| |\gamma (t_q(\xi_0) - x)|} + \sum_q \frac{2 \pi {(\omega_q^{i_1} {(\omega_q)}_{\beta}^{i_2} f_{i_1 i_2})}^2}{|\xi| |\gamma' (t_q(\xi_0)\cdot \xi_0)| |\gamma (t_q(\xi_0) - x)|} = 0.
    \end{equation*}
    By Kirillov-Tuy condition, we get $q = 1, 2, 3$. Hence, we have
    \begin{equation} \label{eq: ellipticity eq 1}
        \omega_q^{i_1} {(\omega_q)}_{\alpha}^{i_2} f_{i_1 i_2} = 0 \quad \text{ and } \quad \omega_q^{i_1} {(\omega_q)}_{\beta}^{i_2} f_{i_1 i_2} = 0; \quad q = 1, 2, 3.
    \end{equation}
    Further, multiplying ${(\omega_1)}_{\alpha}^{i_2}$, ${(\omega_2)}_{\alpha}^{i_2}$ and ${(\omega_1)}_{\beta}^{i_2}$ to equation $\xi^{i_1} f_{i_1 i_2} = 0$ and adding, we get three more equations as follows:
    \begin{equation} \label{eq: ellipticity eq 2}
        \xi^{i_1} {(\omega_1)}_{\alpha}^{i_2} f_{i_1 i_2} = 0, \quad \xi^{i_1} {(\omega_2)}_{\alpha}^{i_2} f_{i_1 i_2} = 0 \quad \mbox{ and } \quad \xi^{i_1} {(\omega_1)}_{\beta}^{i_2} f_{i_1 i_2} = 0.
    \end{equation}
    Without loss of generality, we choose a spherical coordinate system such that $\omega_q$ and ${(\omega_q)}_\alpha$ are parallel to the plane $H(x, \xi)$ and ${(\omega_q)}_{\beta}$ is in the direction of $\xi$. Then writing $\omega_q$, ${(\omega_q)}_{\alpha}$ and ${(\omega_q)}_{\beta}$ for $q = 1, 2, 3$ in spherical coordinates, we have
    \begin{align*}
        \omega_q &= (\sin \alpha_q \cos \beta_1, \sin \alpha_q \sin \beta_1, \cos \alpha_q)\\
        {(\omega_q)}_{\alpha} &= (\cos \alpha_q \cos \beta_1, \cos \alpha_q \sin \beta_1, -\sin \alpha_q)\\
        {(\omega_q)}_{\beta} &= (-\sin \beta_1, \cos \beta_1, 0)
    \end{align*}
    with $\alpha_i \neq \alpha_j \pm \pi$ for $i \neq j$. Substituting the above expressions in equations \eqref{eq: ellipticity eq 1} and \eqref{eq: ellipticity eq 2} for certain values of $q$ as required, we get the following sets of equations:
    \begin{align} \label{eq: 1st set}
    \begin{split}
        \sin \alpha_q &\cos \alpha_q \left\{\cos^2 \beta_1 f_{11} + \sin^2 \beta_1 f_{22} - f_{33} + \sin \beta_1 \cos \beta_1 (f_{12} + f_{21}) \right\}\\
        & - \sin^2 \alpha_q \left\{\cos \beta_1 f_{13} + \sin \beta_1 f_{23}\right\} + \cos^2 \alpha_q \left\{\cos \beta_1 f_{31} + \sin \beta_1 f_{32}\right\} = 0; \quad q = 1, 2, 3,
    \end{split}        
    \end{align}
    
    \begin{align} \label{eq: 2nd set}
    \begin{split}
        \sin \alpha_q &\left\{\sin \beta_1 \cos \beta_1 (f_{11} - f_{22}) - \cos^2 \beta_1 f_{12} + \sin^2 \beta_1 f_{21}\right\}\\
        & \hspace{4cm} + \cos \alpha_q \left\{\sin \beta_1 f_{31} - \cos \beta_1 f_{32}\right\} = 0; \quad q = 1, 2,
    \end{split}
    \end{align}

    \begin{align} \label{eq: 3rd set}
    \begin{split}
        - \cos \alpha_q &\left\{\sin \beta_1 \cos \beta_1 (f_{11} - f_{22}) + \sin^2 \beta_1 f_{12} - \cos^2 \beta_1 f_{21}\right\}\\
        & \hspace{4cm} + \sin \alpha_q \left\{\sin \beta_1 f_{13} - \cos \beta_1 f_{23}\right\} = 0; \quad q = 1, 2,
    \end{split}
    \end{align}
    and
    \begin{align} \label{eq: 4}
        \sin^2 \beta_1 f_{11} + \cos^2 \beta_1 f_{22} - \sin \beta_1 \cos \beta_1 (f_{12} + f_{21}) = 0.
    \end{align}
    System of equations \eqref{eq: 3rd set} can be written in the matrix form $Ax = O$ with 
    \begin{equation*}
        A = \begin{bmatrix}
        -\cos \alpha_1 & \sin \alpha_1\\
        -\cos \alpha_2 & \sin \alpha_2
        \end{bmatrix},
        \quad x = \begin{bmatrix}
            \sin \beta_1 \cos \beta_1 (f_{11} - f_{22}) + \sin^2 \beta_1 f_{12} - \cos^2 \beta_1 f_{21}\\
            \sin \beta_1 f_{13} - \cos \beta_1 f_{23}
        \end{bmatrix} \mbox{   and   }
        O = \begin{bmatrix}
            0\\
            0
        \end{bmatrix}.
    \end{equation*}
    Now, if $\det A = \sin(\alpha_1 - \alpha_2) = 0$, then $\alpha_1 = \alpha_2 \pm n \pi$ which is a contradiction. Hence $\det A \neq 0$ which further gives $x = O$. Applying similar technique to systems of equations \eqref{eq: 1st set} and \eqref{eq: 2nd set} separately, we get $By = O$ and $Bz = O$, where 
    \begin{equation*}
        B = \begin{bmatrix}
            \sin \beta_1 & -\cos \beta_1\\
            \cos \beta_1 & \sin \beta_1
        \end{bmatrix},
        \quad y = \begin{bmatrix}
            f_{13}\\
            f_{23}
        \end{bmatrix} \quad \mbox{       and       }
        \quad z = \begin{bmatrix}
            f_{31}\\
            f_{32}
        \end{bmatrix},
    \end{equation*}
    along with the following equations:
    \begin{align} \label{eq: 6}
    \begin{split}
        \sin \beta_1 \cos \beta_1 (f_{11} - f_{22}) + \sin^2 \beta_1 f_{12} - \cos^2 \beta_1 f_{21} = 0,\\
        \sin \beta_1 \cos \beta_1 (f_{11} - f_{22}) - \cos^2 \beta_1 f_{12} + \sin^2 \beta_1 f_{21} = 0,
    \end{split}
    \end{align}
    \begin{align} \label{eq: 10}
        \cos^2 \beta_1 f_{11} + \sin^2 \beta_1 f_{22} - f_{33} + \sin \beta_1 \cos \beta_1 (f_{12} + f_{21}) = 0,
    \end{align}
    \noindent Since $\det B = 1 \neq 0$, we get
    \begin{align}
        f_{13} = f_{31} = f_{23} = f_{32} = 0.
    \end{align}
    Subtracting the two equations in \eqref{eq: 6} gives $f_{12} = f_{21}$. Using this and adding equations \eqref{eq: 4} and \eqref{eq: 10} further gives
    \begin{align}
        f_{11} + f_{22} - f_{33} = 0.
    \end{align}
    The above equation along with the trace-free condition implies $f_{11} + f_{22} = 0$ and $f_{33} = 0$. Substituting $f_{22} = -f_{11}$ and $f_{12} = f_{21}$ to equations \eqref{eq: 4} and \eqref{eq: 6} and using the matrix determinant technique as above, we finally get $f_{11} = f_{12} = 0$, which further gives $f_{21} = f_{22} = 0$. This completes the proof.    
    \end{proof}
\section{Microlocal Inversion}\label{sec: microlocal_inversion}
In this section, we prove our main theorem by constructing a relative left parametrix for the normal operator $\mathcal{N}$ of the mixed ray transform. 
\begin{proof}[Proof of theorem \ref{thm: main theorem}]
    The proof of Proposition \ref{prop: ellipticity proposition} shows that the set of tensors 
    \begin{equation} \label{eq: LI tensors}
        \left\{\left\{\omega_q \otimes {(\omega_q)}_\alpha; q = 1, 2, 3\right\}, \left\{\omega_q \otimes {(\omega_q)}_\beta; q = 1, 2\right\}\right\}
    \end{equation}
    is linearly independent. Let $V$ be a matrix having tensors 
    $$\left\{\left\{\omega_q \otimes {(\omega_q)}_\alpha; q = 1, 2, 3\right\}, \left\{\omega_q \otimes {(\omega_q)}_\beta; q = 1, 2, 3\right\}\right\}$$
    as column vectors. Then we can write $A_0 (x, \xi) = V V^{t}$. The rank of $A_0 (x, \xi)$ will be 5 due to linear independence of tensors in set \eqref{eq: LI tensors}. Using Singular Value Decomposition for $A_0$ matrix, we can write $A_0 = U D V^{t}$, where $U$ and $V$ are orthogonal matrices and $D$ has 5 non-zero entries. Let $D^-$ be a diagonal matrix obtained by taking reciprocals of non-zero entries of $D$ and let $B_0 (x, \xi) = \sigma(f^{s}) V D^- U^t$ with $\sigma(f^{s})$ being the principal symbol of the solenoidal part $f^{s}$ of $f$. Then, we have
    \begin{equation*}
        B_0 (x, \xi) A_0 (x, \xi) = \sigma(f^{s}) V \begin{pmatrix}
            I_5 & 0\\
            0 & 0
        \end{pmatrix} V^t = \sigma(f^{s})
    \end{equation*}
    Further, define $b_0$ as follows:
    \begin{equation*}
        b_0 (x, \xi) = \begin{cases*}
            B_0 (x, \xi); \quad (x, \xi) \in \Xi_0\\
            0; \hspace{1.5cm} \text{ otherwise }
        \end{cases*}
    \end{equation*}
    Let $\mathcal{B}_0$ be the operator with symbol matrix $b_0$. The possible singularities of $D^-$ are only on $\Sigma$, and hence the entries of $B_0(x, \xi)$ lie in $I(\Delta, \Lambda)$. Now $A_0$ is the symbol of a pseudodifferential operator of order $-1$ away from $\Sigma$. Since $B_0$ is obtained by inversion of $A_0$, $B_0$ is a symbol of order $1$. Hence, we get $\mathcal{B}_0 \in I^{0, 1} (\Delta, \Lambda)$. Also, we have that the operator $\mathcal{N} \in I^{-1, 0} (\Delta, \Lambda)$. Hence, using composition calculus stated in Theorem \ref{Uhlmann_composition_calculus}, we finally get $\mathcal{B}_0 \mathcal{N} \in I^{-\frac{1}{2}, \frac{1}{2}}$.

    \noindent Define $\mathcal{M}_1 = \mathcal{B}_0 \mathcal{N} - f^{s}$. The symbol calculus for $I^{p, l} (\Delta, \Lambda)$ classes is given as follows \cite{Guillemin_Uhlmann_Oscillatory_integrals}:
    \begin{equation*}
        0 \rightarrow I^{p, l - 1} (\Delta, \Lambda) + I^{p - 1, l} (\Delta, \Lambda) \rightarrow I^{p, l} (\Delta, \Lambda) \xrightarrow{\sigma_0} S^{p, l} (\Delta, \Sigma) \rightarrow 0,
    \end{equation*}
    where $S^{p, l} (\Delta, \Sigma)$ denotes the space of product type symbols; for more details, see \cite{Microlocal_2018}. Using this sequence, we can decompose $\mathcal{M}_1$ as $\mathcal{M}_1 = \mathcal{M}_{1 1} + \mathcal{M}_{1 2}$ with $\mathcal{M}_{1 1} \in I^{-\frac{3}{2}, \frac{1}{2}}$ and $\mathcal{M}_{1 2} \in I^{-\frac{1}{2}, -\frac{1}{2}}$. Let $m_{1 1}$ and $m_{1 2}$ be the matrices such that $\sigma_0 (\mathcal{M}_{1 j}) = m_{1 j} A_0;$$ j = 1, 2$. Further, let $\mathcal{B}_{1 1}$ and $\mathcal{B}_{1 2}$ be the operators having symbols $- m_{1 1}$ and $- m_{1 2}$ respectively. For $\mathcal{B}_1 = \mathcal{B}_{1 1} + \mathcal{B}_{1 2}$, let $\mathcal{M}_2 = (\mathcal{B}_0 + \mathcal{B}_1) \mathcal{N} - f^{s}$. Then, we have
    \begin{align*}
        \mathcal{M}_2 &= (\mathcal{B}_0 + \mathcal{B}_1) \mathcal{N} - f^{s}\\
        &= \mathcal{B}_{1 1} \mathcal{N} + \mathcal{B}_{1 2} \mathcal{N} + \mathcal{B}_0 \mathcal{N} - f^{s}\\
        &= \underbrace{\mathcal{B}_{1 1} \mathcal{N} + \mathcal{M}_{1 1}}_{K_1} + \underbrace{\mathcal{B}_{1 2} \mathcal{N} + \mathcal{M}_{1 2}}_{K_2}. 
    \end{align*}
    By construction, we have $K_1 \in I^{-\frac{3}{2}, \frac{1}{2}}$ and $K_2 \in I^{-\frac{1}{2}, -\frac{1}{2}}$ with $\sigma_0 (K_1) = 0 = \sigma_0 (K_2)$. Hence, using symbol calculus, we get:

    \begin{align*}
        K_1 &= K_{1 1} + K_{1 2}, \quad K_{1 1} \in I^{-\frac{5}{2}, \frac{1}{2}}, K_{1 2} \in I^{-\frac{3}{2}, -\frac{1}{2}}\\
        K_2 &= K_{2 1} + K_{2 2}, \quad K_{2 1} \in I^{-\frac{3}{2}, -\frac{1}{2}}, K_{2 2} \in I^{-\frac{1}{2}, -\frac{3}{2}}.
    \end{align*}
    This gives
    \begin{equation*}
        \mathcal{M}_2 = \underbrace{K_{1 1}}_{\mathcal{M}_{2 0}} + \underbrace{K_{1 2} + K_{2 1}}_{\mathcal{M}_{2 1}} + \underbrace{K_{2 2}}_{\mathcal{M}_{2 2}}; \quad \mathcal{M}_{2 0} \in I^{-\frac{5}{2}, \hspace{2mm} \frac{1}{2}}, \mathcal{M}_{2 1} \in I^{-\frac{3}{2}, -\frac{1}{2}}, \hspace{2mm} \mathcal{M}_{2 2} \in I^{-\frac{1}{2}, -\frac{3}{2}}.
    \end{equation*}
    Therefore, we have
    \begin{equation*}
        \mathcal{M}_2 \in \sum_{j = 0}^2 I^{-\frac{1}{2} - 2 + j, \frac{1}{2} - j}.
    \end{equation*}
    Proceeding recursively, we get a sequence of operators
    \begin{equation*}
        \mathcal{M}_N = \sum_{j = 0}^{\left[\frac{N}{2}\right]} I^{-\frac{1}{2} - N + j, \frac{1}{2} - j} + \sum_{\left[\frac{N}{2}\right] + 1}^N I^{-\frac{1}{2} - N + j, \frac{1}{2} - j}.
    \end{equation*}
    Using the inequalities $-\frac{1}{2} - N + j \leq -\frac{1}{2} - N + \left[\frac{N}{2}\right]$, $\frac{1}{2} - j \leq \frac{1}{2}$ in the first summation and $-\frac{1}{2} - N + j \leq -\frac{1}{2}$, $\frac{1}{2} - j \leq -\frac{1}{2} - \left[\frac{N}{2}\right]$ for the second summation, along with the property $I^{p, l} \subset I^{p', l'}$, for $p \leq p', l \leq l'$, we get
    \begin{equation*}
        \sum_{j = 0}^{\left[\frac{N}{2}\right]} I^{-\frac{1}{2} - N + j, \frac{1}{2} - j} \in I^{\frac{1}{2} - N + \left[\frac{N}{2}\right], \frac{1}{2}} \quad \text{ and } \sum_{\left[\frac{N}{2}\right] + 1}^N I^{-\frac{1}{2} - N + j, \frac{1}{2} - j} \in I^{-\frac{1}{2}, -\frac{1}{2} - \left[\frac{N}{2}\right]}.
    \end{equation*}
    Taking limit $N \rightarrow \infty$ and using the properties $\cap_p I^{p, l} \subset C^\infty$ and $\cap_l I^{p, l} (\Delta, \Lambda) \subset I^p(\Lambda)$, we get that the first term in the above expression is a smoothing term and the second expression is an operator $\mathcal{A} \in I^{-\frac{1}{2}} (\Lambda)$. Finally, defining $\mathcal{B} = \mathcal{B}_0 + \mathcal{B}_1 + \mathcal{B}_2 + \dots$, we get
    \begin{equation*}
        \mathcal{B} \mathcal{N} f (x) = f^{s} (x) + \mathcal{A} f (x) + C^\infty \text{ term}.
    \end{equation*}
    This completes the proof of Theorem \ref{thm: main theorem}.
\end{proof}
\section*{Acknowledgements}\label{sec:acknowledge}
I am thankful to Rohit Kumar Mishra and Suman Kumar Sahoo for their valuable suggestions on this work. Also, I acknowledge the financial support due to the PMRF fellowship from the government of India.
\bibliographystyle{plain}
\bibliography{reference}
\end{document}